\documentclass[11pt]{article}%
\usepackage{amsfonts}
\usepackage{amsmath}
\usepackage{amssymb}
\usepackage{euler}
\usepackage{graphicx}%
\providecommand{\U}[1]{\protect\rule{.1in}{.1in}}
\newtheorem{theorem}{Theorem}

\newtheorem{lemma}[theorem]{Lemma}

\newenvironment{proof}[1][Proof]{\noindent\textbf{#1.} }{\ \rule{0.5em}{0.5em}}
\begin{document}

\title{On \textsf{Constructive Connectedness Properties}}
\author{\textsf{Douglas S. Bridges}\\\textsf{School of Mathematics \& Statistics,}\\\textsf{\ University of Canterbury,}\\\textsf{\ Christchurch 8140, New Zealand}}
\maketitle

\begin{abstract}%
%TCIMACRO{\TeXButton{noindent}{\noindent}}%
%BeginExpansion
\noindent
%EndExpansion
\textsf{We plug two gaps in the constructive proof of Theorem 1 (respectively,
Theorem 2) in \cite{dsb}, showing that the property of C-connectedness
(respectively, O-connectedness) of a subset }$S$\textsf{\ of }$\mathbb{R}%
$\textsf{\ is equivalent to}$\ S$\textsf{\ containing the interval }$\left[
a,b\right]  $\textsf{\ (respectively, }$\left(  a,b\right)  $)
\textsf{whenever }$a,b\in S$\textsf{\ and }$a<b$\textsf{. }

\end{abstract}

%

%TCIMACRO{\TeXButton{sf}{\normalfont\sf}}%
%BeginExpansion
\normalfont\sf
%EndExpansion%
%TCIMACRO{\TeXButton{depth 1}{\setcounter{secnumdepth}{0}}}%
%BeginExpansion
\setcounter{secnumdepth}{0}%
%EndExpansion

\subsection{Types of connectedness}

This note amplifies and extends earlier work in the Bishop-style
constructive\footnote{%
%TCIMACRO{\TeXButton{sf}{\normalfont\sf}}%
%BeginExpansion
\normalfont\sf
%EndExpansion
For background in this variety of constructive mathematics, see
\cite{Bishop,BB,bv,Handbook}, of which the last is the most up-to-date
reference.} theory of connectedness \cite{dsb,dsb78a,dsbcon77,MarkM78,MarkM79}%
, in particular for subsets of the real line $\mathbb{R}$.

The proof of Theorem 1 in \cite{dsb} is incomplete in two ways: one, it claims
without justification that if the located set $S$ is C-connected, $x_{0}%
\in\left[  a,b\right]  $, and $a<x_{0}$, then $A\equiv S\cap(-\infty,x_{0}]$
is located; and two,\ the argument establishing the converse lacks detail. In
order to fill these gaps in the proof, we need to access the following definitions.

A metric space $(X,\rho)$ is said to be

\begin{itemize}
\item C-\textsc{connected} if for each located closed subset $A$ of $X$ whose
metric complement%
\[
X-A\equiv\left\{  x\in X:\rho(x,S)>0\right\}
\]
is inhabited, there exists a point in the intersection of $A$ and the closure
$\overline{X-A}\ $of $X-A$;

\item D-\textsc{connected} if $A=X$ whenever $A$ is an inhabited, open,
closed, and located subset of $X$;\footnote{%
%TCIMACRO{\TeXButton{sf}{\normalfont\sf}}%
%BeginExpansion
\normalfont\sf
%EndExpansion
What we name \emph{D-connected} here is called \emph{connected} in
\cite{dsb,dsb78a,dsbcon77}. This change reflects that our current notion of
\emph{connected} is now, and rightly, the standard one in constructive
analysis.} and

\item \textsc{connected} if $A\cap B$ is inhabited whenever $A,B$ are
inhabited, open sets with $X=A\cup B$.
\end{itemize}

%

%TCIMACRO{\TeXButton{noindent}{\noindent}}%
%BeginExpansion
\noindent
%EndExpansion
Note that these three types of connectedness are classically equivalent.
Constructively, C-connected implies D-connected \cite[2.1]{dsbcon77}, but the
converse is essentially nonconstructive \cite[Thms 3 and 4]{dsb}.

\begin{lemma}
\label{jun03l3}Let $S$ be a located D-connected subset of $\mathbb{R}$, and
$a,b$ points of $S$ with $a\leq b$. Then $S\cap\left[  a,b\right]  $ is dense
in $\left[  a,b\right]  \ $\emph{\cite[3.2]{dsbcon77}}.
\end{lemma}

\begin{lemma}
\label{jun03l2}Let $S$ be a located subset of $\mathbb{R}$, let $a,b$ be
points of $S$ with $a<b$ such that $S\cap\left[  a,b\right]  $ is dense in
$S$, and let $x\in\left[  a,b\right]  $. Then $S\cap(-\infty,x]$ and
$S\cap\lbrack x,\infty)$ are totally bounded.
\end{lemma}

\begin{proof}
Since $S$ is located in $\mathbb{R}$, it is locally totally bounded
\cite[2.2.18]{bv}. Hence for all but countably many $r>0$, the set
$S\cap\left[  a-r,a+r\right]  $ is totally bounded (see Definition 6.1 in
Chapter 4 of \cite{BB}). It will suffice to consider $r>b-a$ such that
$S\cap\left[  a-r,a+r\right]  $ is totally bounded and to prove that%
\[
A\equiv S\cap(-\infty,x]\cap\left[  a-r,a+r\right]
\]
is totally bounded. To that end, first note that $\left[  a,b\right]
\subset\left[  a-r,a+r\right]  $. Let $\varepsilon>0$ and construct an
$\varepsilon/3$-approximation $\left\{  \xi_{1},\ldots,\xi_{N}\right\}  $ to
$S\cap\left[  a-r,a+r\right]  $. Write $\left\{  1,\ldots,N\right\}  $ as a
union of sets $P$ and $Q$, where

\begin{itemize}
\item[--] if $k\in P$, then $\xi_{k}<x$ and therefore $\xi_{k}\in A$, and

\item[--] if $k\in Q$, then $\xi_{k}>x-\varepsilon/3$.
\end{itemize}

%

%TCIMACRO{\TeXButton{noindent}{\noindent}}%
%BeginExpansion
\noindent
%EndExpansion
Consider any $k\in Q$. Note that for each $z\in A$ with $\left\vert z-\xi
_{k}\right\vert <\varepsilon/3$,
\[
\xi_{k}-\frac{\varepsilon}{3}<z\leq x<\xi_{k}+\frac{\varepsilon}{3}%
\]
and therefore $0<x-z<2\varepsilon/3$. Either $x>a$ or $x<a+\varepsilon/3$. In
the first case, since $S\cap\left[  a,b\right]  $ is dense in $\left[
a,b\right]  $, we can pick $\xi_{k}^{\prime}\in S$ with $a<\xi_{k}^{\prime}<x
$ and $x-\xi_{k}^{\prime}<\varepsilon/3$; then $\xi_{k}^{\prime}\in A$, and
for each $z\in A$ with $\left\vert z-\xi_{k}\right\vert <\varepsilon/3$,%
\begin{equation}
\left\vert z-\xi_{k}^{\prime}\right\vert \leq\left\vert z-x\right\vert
+\left\vert x-\xi_{k}^{\prime}\right\vert <\frac{2\varepsilon}{3}%
+\frac{\varepsilon}{3}=\varepsilon. \label{2a}%
\end{equation}
In the second case, $a\leq$ $x<a+\varepsilon/3$. Setting $\xi_{k}^{\prime
}=a\in A$, for each $z\in A$ with $\left\vert z-\xi_{k}\right\vert
<\varepsilon/3$ we again have (\ref{2a}). It now follows that
\[
\left\{  \xi_{k}:k\in P\right\}  \cup\left\{  \xi_{k}^{\prime}:k\in Q\right\}
\]
is a finitely enumerable $\varepsilon$-approximation to $A$. Since
$\varepsilon>0$ is arbitrary, we conclude that $A$ is totally bounded.
\end{proof}

%

%TCIMACRO{\TeXButton{medskip}{\medskip}}%
%BeginExpansion
\medskip
%EndExpansion

We can now fill the gaps in the original argument for \cite[Thm 1]{dsb}.

\begin{theorem}
\label{jun03t1}The following conditions are equivalent on a located subset $S
$ of $\mathbb{R}$.

\begin{enumerate}
\item[\emph{(i)}] $S$ is C-connected.

\item[\emph{(ii)}] If $a,b\in S$ and $a<b$, then $\left[  a,b\right]  \subset
S $.
\end{enumerate}
\end{theorem}

\begin{proof}
Assuming (i), and given $x\in\left[  a,b\right]  $, we have either $a<x$ or
$x<b$. Without loss of generality, we assume the latter alternative. Being
C-connected, $S$ is D-connected, so by Lemma \ref{jun03l3}, $S\cap\left[
a,b\right]  $ is dense in $\left[  a,b\right]  $. Hence, by Lemma
\ref{jun03l2}, the set $A\equiv S\cap(-\infty,x]$ is locally totally bounded
and therefore located in $\mathbb{R}$ \cite[2.2.18]{bv}. Moreover, $A$ is
clearly closed in $S$, and $\rho(b,A)\geq b-x>0$. Since $S$ is C-connected,
there exists $\xi\in A\cap\overline{S-A}$. Clearly, $\xi\leq x$. If $\xi<x$,
then there exists $s\in S-A$ with $\xi<s<x$ and therefore $x\in A$, a
contradiction. Hence $\xi\geq x$ and therefore $\xi=x$. Thus (i) implies (ii).

Now assume (ii) and let $A$ be a located closed subset of $S$. Let $a\in A$
and let $S-A$ be inhabited by a point $b$; then $a\neq b$.\footnote{%
%TCIMACRO{\TeXButton{sf}{\normalfont\sf}}%
%BeginExpansion
\normalfont\sf
%EndExpansion
Reader, be aware that in a metric space $\left(  X,\rho\right)  $, the
statement $x\neq y$ means, \emph{not }$\lnot(x=y)$, but $\rho(x,y)>0$.}
Consider, for example, the case $a<b$. Setting $a_{1}=a$ and $b_{1}=b$, we
construct sequences $\left(  a_{n}\right)  _{n\geq1},\left(  b_{n}\right)
_{n\geq1}$ in $\left[  a,b\right]  $ such that for each $n>1$,%
\[
0<b_{n}-a_{n}<\tfrac{3}{4}(b_{n-1}-a_{n-1}),
\]
$a_{n}\in A$, $b_{n}\in S-A$, and $\left[  a_{n},b_{n}\right]  \subset\left[
a_{n-1},b_{n-1}\right]  $. Suppose we have found $a_{n-1}$ and $b_{n-1}$ with
the applicable properties. Let $c_{n}=\frac{1}{2}\left(  a_{n-1}%
+b_{n-1}\right)  $. Either $\rho(c_{n},A)>0$ or $\rho(c_{n},A)<\frac{1}%
{4}\left(  b_{n-1}-a_{n-1}\right)  $. In the first case, set $a_{n}=a_{n-1}\in
A$ and $b_{n}=c_{n}\in S-A$, and note that%
\[
0<b_{n}-a_{n}=\tfrac{1}{2}\left(  b_{n-1}-a_{n-1}\right)  <\tfrac{3}{4}\left(
b_{n-1}-a_{n-1}\right)  .
\]
In the second case, choose $a_{n}\in A$ such that $\left\vert a_{n}%
-c_{n}\right\vert <\frac{1}{4}\left(  b_{n-1}-a_{n-1}\right)  $ and set
$b_{n}=b_{n-1}\in S-A$; then%
\[
0<\tfrac{1}{4}\left(  b_{n-1}-a_{n-1}\right)  <b_{n}-a_{n}<\tfrac{3}{4}\left(
b_{n-1}-a_{n-1}\right)  .
\]
In both cases we clearly have $\left[  a_{n},b_{n}\right]  \subset\left[
a_{n-1},b_{n-1}\right]  $. This completes the inductive construction. We now
see that $\left(  \left[  a_{n},b_{n}\right]  \right)  _{n\geq1}$ is a
descending sequence of intervals with diameters tending to $0$ as
$n\rightarrow\infty$. Hence $\left(  a_{n}\right)  _{n\geq1}$ and $\left(
b_{n}\right)  _{n\geq1}~$both converge to the unique point $\xi$ of $%
%TCIMACRO{\tbigcap _{n\geq1}}%
%BeginExpansion
{\textstyle\bigcap_{n\geq1}}
%EndExpansion
\left[  a_{n},b_{n}\right]  $. Clearly, $\xi\in\overline{A}=A$ and $\xi
\in\overline{S-A}$, so $\xi\in A\cap\overline{S-A}$. Thus (ii) implies (i).
\end{proof}

%

%TCIMACRO{\TeXButton{medskip}{\medskip}}%
%BeginExpansion
\medskip
%EndExpansion

We now look briefly at the remaining notion of connectedness discussed in
\cite{dsb,dsb78a,dsbcon77,MarkM78}.

A metric space $(X,\rho)$ is said to be\ O-\textsc{connected} if for each
located open subset $A$ of $X$ with inhabited metric complement, there exists
a point $\xi$ in the closure of $A$ such that $\xi\neq x$ for each $x\in A$.
The theorem we want to address, whose proof in \cite{dsb} has gaps similar to
the two in the proof of \cite[Thm 1]{dsb}, is the following.

\begin{theorem}
\label{jun06t1}The following conditions are equivalent on a located subset $S
$ of $\mathbb{R}$.

\begin{enumerate}
\item[\emph{(i)}] $S$ is O-connected.

\item[\emph{(ii)}] If $a,b\in S$ and $a<b$, then the open interval $\left(
a,b\right)  \subset S$.
\end{enumerate}
\end{theorem}

\begin{proof}
[Sketch Proof]Assuming (i), let $a,b\in S$ and $a<b$. Since O-connected
implies D-connected \cite[2.1]{dsbcon77}, Lemma \ref{jun03l3} shows that
$S\cap\left[  a,b\right]  $ is dense in $\left[  a,b\right]  $. A proof
similar to that of Lemma \ref{jun03l2} shows that if $x\in\left(  a,b\right)
$, then $A\equiv S\cap(-\infty,x)$ is located in $\mathbf{R}$. Since $A$ is
open in $S$ and $b\in S-A$, it follows from the O-connectedness of $S$ that
there exists $\xi\in S\cap\overline{A}$ such that $\xi\neq z$ for each $z\in
A$. Since $\xi\in\overline{A}$, we have $\xi\leq x$. On the other hand, if
$\xi<x$, then $\xi\in A$, which is absurd. Thus $x=\xi\in S$. Hence (i)
implies (ii).

Suppose, conversely, that (ii) holds, and let $A$ be a located open subset of
$S$. Let $a\in A$ and let $S-A$ be inhabited by a point $b$; then $a\neq b$.
We may assume that $a<b$. Mirroring the argument in the second half of the
proof of Theorem \ref{jun03t1}, we construct $\xi\in\left[  a,b\right]
\cap\overline{A}\cap\overline{(S-A)}$. For each $x\in A$ there exists $r>0$
such that $(x-r,x+r)\subset A$. If $\left\vert \xi-x\right\vert <r$, then we
can find $y\in S-A$ with $\left\vert y-\xi\right\vert <r-\left\vert
\xi-x\right\vert $ and therefore $\left\vert y-x\right\vert <r$; but then
$y\in A\cap(S-A)$, a contradiction. We conclude that $\xi\neq x$ for each
$x\in A$. In particular, $\xi\neq a$ and $\xi\neq b$, so $\alpha<\xi<b$ and
therefore $\xi\in S$. Thus (ii) implies (i).
\end{proof}

\subsection{Concluding remarks}

Does \emph{connected} imply either \emph{C-connected} or \emph{O-connected}?
One step in the direction of an affirmative answer to these questions might be this.

\begin{lemma}
\label{jun03l1}Let $S$ be a connected subset of $\mathbb{R}$, and $a,b$ points
of $S$ with $a<b$. Then $S\cap\left(  a,b\right)  $ is dense in $\left[
a,b\right]  $.
\end{lemma}

\begin{proof}
Let $x\in\left(  a,b\right)  $ and $\varepsilon>0$. Then $S$ is the union of
the two subsets $S\cap(-\infty,x+\varepsilon)$ and $S\cap\left(
x-\varepsilon,\infty\right)  $, which are inhabited and open in $S$. Hence
there is a point $y$ of $S$ that belongs to the intersection of those two
sets. Then $x-\varepsilon<y<x+\varepsilon$, so $\left\vert x-y\right\vert
<\varepsilon$.
\end{proof}%

%TCIMACRO{\TeXButton{medskip}{\medskip}}%
%BeginExpansion
\medskip
%EndExpansion

However, since continuous mappings preserve connectedness, the Brouwerian
example\footnote{%
%TCIMACRO{\TeXButton{sf}{\normalfont\sf}}%
%BeginExpansion
\normalfont\sf
%EndExpansion
That Brouwerian example was originally presented by Bishop on page 5 of
\cite{Bishop}. For other Brouwerian examples on connectedness properties, see
\cite{dsb,MarkM78}.} on pages 7--8 of \cite{BB} serves to show that if, for
all uniformly continuous functions from $\left[  0,1\right]  $ to$~\left[
-1,1\right]  $, the totally bounded, hence located, and connected range of $f$
had property (ii) of Theorem \ref{jun03t1}, then we could prove the
omniscience principle

\begin{quote}
LLPO: If $\left(  a_{n}\right)  $ is a binary sequence with at most one term
equal to $0$ (that is, with $a_{m}a_{n}=0$ for all distinct $m$ and $n$), then
either $a_{n}=0$ for all even $n$ or else $a_{n}=0$ for all odd $n$.
\end{quote}%

%TCIMACRO{\TeXButton{noindent}{\noindent}}%
%BeginExpansion
\noindent
%EndExpansion
Bearing in mind Theorem \ref{jun06t1}, we now see that a constructive proof
that \emph{connected} implies \emph{O-connected} would provide us with one of
LLPO. Hence that implication---and \emph{a fortiori }that from \emph{connected
}to \emph{C-connected---}is essentially nonconstructive.%

%TCIMACRO{\TeXButton{bigskip}{\bigskip}}%
%BeginExpansion
\bigskip
%EndExpansion
%

%TCIMACRO{\TeXButton{bigskip}{\bigskip}}%
%BeginExpansion
\bigskip
%EndExpansion
%

%TCIMACRO{\TeXButton{noindent}{\noindent}}%
%BeginExpansion
\noindent
%EndExpansion
\textbf{Funding: \ \ }Not applicable.%

%TCIMACRO{\TeXButton{noindent}{\noindent}}%
%BeginExpansion
\noindent
%EndExpansion
\textbf{Author contributions: }\ I am the sole author of the paper, which is
100\% due to me.

%

%TCIMACRO{\TeXButton{bigskip}{\bigskip}}%
%BeginExpansion
\bigskip
%EndExpansion
%

%TCIMACRO{\TeXButton{bigskip}{\bigskip}}%
%BeginExpansion
\bigskip
%EndExpansion
%

%TCIMACRO{\TeXButton{noindent}{\noindent}}%
%BeginExpansion
\noindent
%EndExpansion
\textbf{Author's email: }\texttt{dugbridges@gmail.com}

\end{document}